\documentclass{amsart}
\usepackage{amssymb}
\usepackage{amsmath}
\usepackage{amsfonts}

\theoremstyle{plain}

\newtheorem{corollary}{Corollary}

\newtheorem{definition}{Definition}
\newtheorem{example}{Example}

\newtheorem{lemma}{Lemma}

\newtheorem{proposition}{Proposition}
\newtheorem{remark}{Remark}

\numberwithin{equation}{section}
\input{tcilatex}

\begin{document}
\title[Pairwise Weakly Lindel\"{o}f Spaces]{On The Pairwise Weakly Lindel\"{o}f
Bitopological Spaces}

\author{ADEM KILI\c{C}MAN}
\address{Department of Mathematics, University Putra Malaysia,
43400 UPM, Serdang, Selangor, Malaysia.}
\email{akilicman@putra.upm.edu.my}
\author{ZABIDIN SALLEH}
\address{Department of Mathematics, University Malaysia
Terengganu, 21030 Kuala Terengganu, Terengganu, Malaysia.}
\email{bidisalleh@yahoo.com}

\date{\today}

\begin{abstract}
In this paper we shall introduce and study the notion of pairwise
weakly Lindel\"{o}f bitopological spaces and obtained some
results. Further, we also study the pairwise weakly Lindel\"{o}f
subspaces and subsets, investigate some of their properties and
show that a pairwise weakly Lindel\"{o}f property is not a
hereditary property.
\end{abstract}

\maketitle

\section{\textbf{Introduction}}

\noindent In literature there are several generalizations of the notion of Lindel\"{o}%
f spaces and these are studied separately for different reasons
and purposes. In 1959, Frolik \cite{Frol59} introduced and studied
the notion of weakly Lindel\"{o}f spaces that, afterward, was
studied by several authors. Thereafter in 1996, Cammaroto and
Santoro \cite{CammaSan96} studied and gave further new results
about these spaces.\\

\noindent The purpose of this paper is to define the notion of
weakly Lindel\"{o}f
spaces in bitopological spaces, which we will call pairwise weakly Lindel%
\"{o}f spaces and investigate some of their properties. Moreover
we study the pairwise weakly Lindel\"{o}f subspaces and subsets
and also investigate some of their properties. \\

\noindent In section $3$, we shall introduce the concept of pairwise weakly Lindel\"{o}%
f bitopological spaces by investigating the $ij$-weakly Lindel\"{o}f
property and obtain some results. Furthermore, we study the relation between
$ij$-almost Lindel\"{o}f spaces and $ij$-weakly Lindel\"{o}f spaces, where $%
i,j=1$ or $2$, $i\neq j$.\\

\noindent In section $4$, we shall define the concept of pairwise
weakly Lindel\"{o}f subspaces and subsets in bitopological spaces.
We shall define the concept of pairwise weakly Lindel\"{o}f
relative to a bitopological space by investigating the $ij$-weakly
Lindel\"{o}f property and obtain some results. The main result we
are obtain here is pairwise weakly Lindel\"{o}f property is not a
hereditary property by a counter-example obtained.

\section{\textbf{Preliminaries}}

Throughout this paper, all spaces $\left( X,\tau \right) $ and $\left(
X,\tau _{1},\tau _{2}\right) $ $($or simply $X)$ are always mean topological
spaces and bitopological spaces, respectively unless explicitly stated. By $%
i $-open set, we shall mean the open set with respect to topology
$\tau _{i}$ in $X$. We always use $ij$- to denote the certain
properties with respect to topology $\tau _{i}$ and $\tau _{j}$
respectively, where $i,j\in \left\{ 1,2\right\} $ and $i\neq j$.
In this paper, every result in terms of $ij$- will have pairwise
as a corollary.\\

\noindent By $i$-$\limfunc{int}\left( A\right) $ and $i$-$\limfunc{cl}\left( A\right) $%
, we shall mean the interior and the closure of a subset $A$ of $X$ with
respect to topology $\tau _{i}$, respectively. We denote by $\limfunc{int}%
\left( A\right) $ and $\limfunc{cl}\left( A\right) $ for the interior and
the closure of a subset $A$ of $X$ with respect to topology $\tau _{i}$ for
each $i=1,2$, respectively. By $i$-open cover of $X$, we mean that the cover
of $X$ by $i$-open sets in $X$; similar for the $ij$-regular open cover of $%
X $ and etc.\\

\noindent If $S\subseteq A\subseteq X$, then
$i$-$\limfunc{int}_{A}\left( S\right) $ and
$i$-$\limfunc{cl}_{A}\left( S\right) $ will be used to denote the
interior and closure of $S$ in the subspace $A$ with respect to topology $%
\tau _{i}$ respectively. This means that $i$-$\limfunc{int}\left( S\right)
=i $-$\limfunc{int}\left( S\right) \cap A$ and $i$-$\limfunc{cl}_{A}\left(
S\right) =i$-$\limfunc{cl}\left( S\right) \cap A$.

\begin{definition}[see \protect\cite{KhedShib91, SingAry71}]
A subset $S$ of a bitopological space $\left( X,\tau _{1},\tau _{2}\right) $
is said to be $ij$-regular open $($resp. $ij$-regular closed$)$ if $i$-$%
\limfunc{int}\left( j\text{-}\limfunc{cl}\left( S\right) \right) =S$ $($%
resp. $i$-$\limfunc{cl}\left( j\text{-}\limfunc{int}\left( S\right) \right)
=S)$, where $i,j\in \left\{ 1,2\right\} $ and $i\neq j$. $S$ is said
pairwise regular open $($resp. pairwise regular closed$)$ if it is both $ij$%
-regular open and $ji$-regular open $($resp. $ij$-regular closed and $ji$%
-regular closed$)$.
\end{definition}

\noindent The bitopology $\sigma _{i\text{ }}$generated by the
$ij$-regular open subsets of $\left( X,\tau _{1},\tau _{2}\right)
$ is denoted by $\tau _{ij}^{s}$ and it is called
$ij$-semiregularization of $X$. If $\sigma _{i}$ generated by
$ji$-regular open subsets of $\left( X,\tau _{1},\tau
_{2}\right) $, then it is called $ji$-semiregularization of $X$ denoted by $%
\tau _{ji}^{s}$. The bitopology $\sigma _{i}$ is called pairwise
semiregularization of $X$ if it is both $ij$-semiregularization of $X$ and $%
ji$-semiregularization of $X$ denoted by $\tau _{i}^{s}$. If $\tau
_{i}\equiv \tau _{ij}^{s}$, then $X$ is said to be $ij$-semiregular and if $%
\tau _{i}\equiv \tau _{ji}^{s}$, then $X$ is said $ji$-semiregular. $X$ is
said pairwise semiregular if it is both $ij$-semiregular and $ji$%
-semiregular, that is, whenever $\tau _{i}\equiv \tau _{i}^{s}$.

\begin{definition}[see \protect\cite{ForaHdeib83}]
\bigskip A bitopological space $\left( X,\tau _{1},\tau _{2}\right) $ is
said to be $i$-Lindel\"{o}f if the topological space $\left( X,\tau
_{i}\right) $ is Lindel\"{o}f. $X$ is said Lindel\"{o}f if it is $i$-Lindel%
\"{o}f for each $i=1,2$.
\end{definition}

\begin{definition}[see \protect\cite{BidiAdem1}]
A bitopological space $X$ is said to be $ij$-nearly Lindel\"{o}f if for
every $i$-open cover $\left\{ U_{\alpha }:\alpha \in \Delta \right\} $ of $X$%
, there exists a countable subset $\left\{ \alpha _{n}:n\in
\mathbb{N}
\right\} $ of $\Delta $ such that $X=\dbigcup\limits_{n\in
\mathbb{N}
}i$-$\limfunc{int}\left( j\text{-}\limfunc{cl}\left( U_{\alpha _{n}}\right)
\right) $, or equivalently, every $ij$-regular open cover of $X$ has a
countable subcover. $X$ is said pairwise nearly Lindel\"{o}f if it is both $%
ij$-nearly Lindel\"{o}f and $ji$-nearly Lindel\"{o}f.
\end{definition}

\begin{definition}[see \protect\cite{KhedShib91,SingSing70}]
A bitopological space $X$ is said to be $ij$-almost regular if for each $%
x\in X$ and for each $ij$-regular open set $V$ of $X$ containing $x$, there
is an $i$-open set $U$ such that $x\in U\subseteq j$-$\limfunc{cl}\left(
U\right) \subseteq V$. $X$ is said pairwise almost regular if it is both $ij$%
-almost regular and $ji$-almost regular.
\end{definition}

\begin{definition}[see \protect\cite{Kelly63, KhedShib91}]
A bitopological space $\left( X,\tau _{1},\tau _{2}\right) $ is said to be $%
ij$-regular if for each point $x\in X$ and for each $i$-open set $V$
containing $x$, there exists an $i$-open set $U$ such that $x\in U\subseteq
j $-$\limfunc{cl}\left( U\right) \subseteq V$. $X$ is said pairwise regular
if it is both $ij$-regular and $ji$-regular.
\end{definition}

\begin{definition}
A bitopological space $X$ is said to be $ij$-semiregular if for each $x\in X$
and for each $i$-open set $V$ of $X$ containing $x$, there is an $i$-open
set $U$ such that%
\begin{equation*}
x\in U\subseteq i\text{-}\limfunc{int}\left( j\text{-}\limfunc{cl}\left(
U\right) \right) \subseteq V.
\end{equation*}%
$X$ is said pairwise semi regular if it is both $ij$-semiregular and $ji$%
-semiregular.
\end{definition}

\begin{definition}
Let $\left( X,\tau _{1},\tau _{2}\right) $ be a bitopological space$:$

$\left( i\right) $ \ A subset $G$ of $X$ is said to be open if $G$ is both $%
1 $-open and $2$-open in $X$, or equivalently, $G\in U$ for $U\subseteq
\left( \tau _{1}\cap \tau _{2}\right) $ in $X$.

$\left( ii\right) $ \ A subset $F$ of $X$ is said to be closed if $F$ is
both $1$-closed and $2$-closed in $X$, or equivalently, $F\in V$ for $%
V=X\setminus U$ and $U\subseteq $ $\left( \tau _{1}\cap \tau _{2}\right) $
in $X$.
\end{definition}

\section{\textbf{Pairwise Weakly Lindel\"{o}f Spaces}}

\begin{definition}
A bitopological space $X$ is said to be $ij$-weakly Lindel\"{o}f if for
every $i$-open cover $\left\{ U_{\alpha }:\alpha \in \Delta \right\} $ of $X$%
, there exists a countable subset $\left\{ \alpha _{n}:n\in
\mathbb{N}
\right\} $ of $\Delta $ such that $X=j$-$\limfunc{cl}\left(
\dbigcup\limits_{n\in
\mathbb{N}
}U_{\alpha _{n}}\right) $. $X$ is said pairwise weakly Lindel\"{o}f if it is
both $ij$-weakly Lindel\"{o}f and $ji$-weakly Lindel\"{o}f.
\end{definition}

\begin{proposition}
A bitopological space $X$ is $ij$-weakly Lindel\"{o}f if and only if every
family $\left\{ C_{\alpha }:\alpha \in \Delta \right\} $ of $i$-closed
subsets of $X$ such that $\dbigcap\limits_{\alpha \in \Delta }C_{\alpha
}=\emptyset $ admits a countable subfamily $\left\{ C_{\alpha _{n}}:n\in
\mathbb{N}
\right\} $ such that $j$-$\limfunc{int}\left( \dbigcap\limits_{n\in
\mathbb{N}
}C_{\alpha _{n}}\right) =\emptyset $.
\end{proposition}

\begin{proof}
Let $\left\{ C_{\alpha }:\alpha \in \Delta \right\} $ is a family of $i$%
-closed subsets of $X$ such that $\dbigcap\limits_{\alpha \in \Delta
}C_{\alpha }=\emptyset $. Then $X=X\setminus $ $\dbigcap\limits_{\alpha \in
\Delta }C_{\alpha }=\dbigcup\limits_{\alpha \in \Delta }\left( X\setminus
C_{\alpha }\right) $, i.e., the family $\left\{ X\setminus C_{\alpha
}:\alpha \in \Delta \right\} $ is an $i$-open cover of $X$. Since $X$ is $ij$%
-weakly Lindel\"{o}f, there exists a countable subfamily $\left\{ X\setminus
C_{\alpha _{n}}:n\in
\mathbb{N}
\right\} $ such that $X=j$-$\limfunc{cl}\left( \dbigcup\limits_{n\in
\mathbb{N}
}\left( X\setminus C_{\alpha _{n}}\right) \right) $. So $X\setminus j$-$%
\limfunc{cl}\left( \dbigcup\limits_{n\in
\mathbb{N}
}\left( X\setminus C_{\alpha _{n}}\right) \right) =\emptyset $, i.e., $j$-$%
\limfunc{int}\left( X\setminus \dbigcup\limits_{n\in
\mathbb{N}
}\left( X\setminus C_{\alpha _{n}}\right) \right) =\emptyset $. Thus
\begin{equation*}
j\text{-}\limfunc{int}\left( \dbigcap\limits_{n\in
\mathbb{N}
}C_{\alpha _{n}}\right) =\emptyset .
\end{equation*}
Conversely, let $\left\{ U_{\alpha }:\alpha \in \Delta \right\} $ be an $i$%
-open cover of $X$. Then $X=\dbigcup\limits_{\alpha \in \Delta }U_{\alpha }$
and $\left\{ X\setminus U_{\alpha }:\alpha \in \Delta \right\} $ is a family
of $i$-closed subsets of $X$. Hence $X\setminus \dbigcup\limits_{\alpha \in
\Delta }U_{\alpha }=\emptyset $, i.e., $\dbigcap\limits_{\alpha \in \Delta
}\left( X\setminus U_{\alpha }\right) =\emptyset $. By hypothesis, there
exists a countable subfamily $\left\{ X\setminus U_{\alpha _{n}}:n\in
\mathbb{N}
\right\} $ such that $j$-$\limfunc{int}\left( \dbigcap\limits_{n\in
\mathbb{N}
}\left( X\setminus U_{\alpha _{n}}\right) \right) =\emptyset $. So%
\begin{equation*}
X=X\setminus j\text{-}\limfunc{int}\left( \dbigcap\limits_{n\in
\mathbb{N}
}\left( X\setminus U_{\alpha _{n}}\right) \right) =j\text{-}\limfunc{cl}%
\left( X\setminus \dbigcap\limits_{n\in
\mathbb{N}
}\left( X\setminus U_{\alpha _{n}}\right) \right) =j\text{-}\limfunc{cl}%
\left( \dbigcup\limits_{n\in
\mathbb{N}
}U_{\alpha _{n}}\right) .
\end{equation*}%
Therefore $X$ is $ij$-weakly Lindel\"{o}f.
\end{proof}

\begin{corollary}
A bitopological space $X$ is pairwise weakly Lindel\"{o}f if and only if
every family $\left\{ C_{\alpha }:\alpha \in \Delta \right\} $ of closed
subsets of $X$ such that $\dbigcap\limits_{\alpha \in \Delta }C_{\alpha
}=\emptyset $ admits a countable subfamily $\left\{ C_{\alpha _{n}}:n\in
\mathbb{N}
\right\} $ such that $\limfunc{int}\left( \dbigcap\limits_{n\in
\mathbb{N}
}C_{\alpha _{n}}\right) =\emptyset $.
\end{corollary}

\begin{proof}
It is obvious by the definition.
\end{proof}

\begin{proposition}
A bitoplogical space $X$ is $ij$-weakly Lindel\"{o}f if and only if for
every family $\left\{ C_{\alpha }:\alpha \in \Delta \right\} $ by $i$-closed
subsets of $X$, there exists a countable subfamily $\left\{ C_{\alpha
_{n}}:n\in
\mathbb{N}
\right\} $ such that $j$-$\limfunc{int}\left( \dbigcap\limits_{n\in
\mathbb{N}
}C_{\alpha _{n}}\right) \neq \emptyset $, the intersection $%
\dbigcap\limits_{\alpha \in \Delta }C_{\alpha }\neq \emptyset $.
\end{proposition}

\begin{proof}
Let $\left\{ C_{\alpha }:\alpha \in \Delta \right\} $ be a family of $i$%
-closed subsets of $X$ with there exists a countable subfamily $\left\{
C_{\alpha _{n}}:n\in
\mathbb{N}
\right\} $ such that $j$-$\limfunc{int}\left( \dbigcap\limits_{n\in
\mathbb{N}
}C_{\alpha _{n}}\right) \neq \emptyset $. Suppose that $\dbigcap\limits_{%
\alpha \in \Delta }C_{\alpha }=\emptyset $. Hence $X=X\setminus $ $%
\dbigcap\limits_{\alpha \in \Delta }C_{\alpha }=\dbigcup\limits_{\alpha \in
\Delta }\left( X\setminus C_{\alpha }\right) $. Thus $\left\{ X\setminus
C_{\alpha }:\alpha \in \Delta \right\} $ forms an $i$-open cover for $X$.
Since $X$ is $ij$-weakly Lindel\"{o}f, there exists a countable subset $%
\left\{ \alpha _{n}:n\in
\mathbb{N}
\right\} $ of $\Delta $ such that $X=j$-$\limfunc{cl}\left(
\dbigcup\limits_{n\in
\mathbb{N}
}\left( X\setminus C_{\alpha _{n}}\right) \right) $. Hence $X\setminus j$-$%
\limfunc{cl}\left( \dbigcup\limits_{n\in
\mathbb{N}
}\left( X\setminus C_{\alpha _{n}}\right) \right) =\emptyset $, i.e., $j$-$%
\limfunc{int}\left( X\setminus \dbigcup\limits_{n\in
\mathbb{N}
}\left( X\setminus C_{\alpha _{n}}\right) \right) =\emptyset $. Thus
\begin{equation*}
j\text{-}\limfunc{int}\left( \dbigcap\limits_{n\in
\mathbb{N}
}C_{\alpha _{n}}\right) =\emptyset
\end{equation*}
which is a contradiction.\\

\noindent We will give two methods to prove the converse of the
Proposition.\\

\noindent \textit{First Method}: Let the condition is holds. By
contrapositive, this
condition is equivalent to the following statement: if for every family $%
\left\{ C_{\alpha }:\alpha \in \Delta \right\} $ by $i$-closed subsets of $X$
such that the intersection $\dbigcap\limits_{\alpha \in \Delta }C_{\alpha
}=\emptyset $, then there exists a countable subfamily $\left\{ C_{\alpha
_{n}}:n\in
\mathbb{N}
\right\} $ such that $j$-$\limfunc{int}\left( \dbigcap\limits_{n\in
\mathbb{N}
}C_{\alpha _{n}}\right) =\emptyset $. Let $\left\{ U_{\alpha }:\alpha \in
\Delta \right\} $ be an $i$-open cover of $X$. Then $X=\dbigcup\limits_{%
\alpha \in \Delta }U_{\alpha }$ and $\left\{ X\setminus U_{\alpha }:\alpha
\in \Delta \right\} $ be a family of $i$-closed subsets of $X$. Hence $%
X\setminus \dbigcup\limits_{\alpha \in \Delta }U_{\alpha }=\emptyset $,
i.e., $\dbigcap\limits_{\alpha \in \Delta }\left( X\setminus U_{\alpha
}\right) =\emptyset $. By hypothesis, there exists a countable subfamily $%
\left\{ X\setminus U_{\alpha _{n}}:n\in
\mathbb{N}
\right\} $ such that $j$-$\limfunc{int}\left( \dbigcap\limits_{n\in
\mathbb{N}
}\left( X\setminus U_{\alpha _{n}}\right) \right) =\emptyset $. So $%
X=X\setminus j$-$\limfunc{int}\left( \dbigcap\limits_{n\in
\mathbb{N}
}\left( X\setminus U_{\alpha _{n}}\right) \right) =j$-$\limfunc{cl}\left(
X\setminus \dbigcap\limits_{n\in
\mathbb{N}
}\left( X\setminus U_{\alpha _{n}}\right) \right) =j$-$\limfunc{cl}\left(
\dbigcup\limits_{n\in
\mathbb{N}
}U_{\alpha _{n}}\right) $. Therefore $X$ is $ij$-weakly
Lindel\"{o}f.\\

\noindent \textit{Second Method}: \ Suppose that $X$ is not
$ij$-weakly Lindel\"{o}f. Then there exists an $i$-open cover
$\left\{ U_{\alpha }:\alpha \in \Delta \right\} $ of $X$ with no
countable subfamily $\left\{ U_{\alpha _{n}}:n\in
\mathbb{N}
\right\} $ such that $X=j$-$\limfunc{cl}\left( \dbigcup\limits_{n\in
\mathbb{N}
}U_{\alpha _{n}}\right) $. Hence $X\neq j$-$\limfunc{cl}\left(
\dbigcup\limits_{n\in
\mathbb{N}
}U_{\alpha _{n}}\right) $ for any countable subfamily $\left\{ U_{\alpha
_{n}}:n\in
\mathbb{N}
\right\} $. It follows that $X\setminus j$-$\limfunc{cl}\left(
\dbigcup\limits_{n\in
\mathbb{N}
}U_{\alpha _{n}}\right) \neq \emptyset $, i.e., $j$-$\limfunc{int}\left(
X\setminus \dbigcup\limits_{n\in
\mathbb{N}
}U_{\alpha _{n}}\right) \neq \emptyset $ or $j$-$\limfunc{int}\left(
\dbigcap\limits_{n\in
\mathbb{N}
}\left( X\setminus U_{\alpha _{n}}\right) \right) \neq \emptyset $. Thus $%
\left\{ X\setminus U_{\alpha }:\alpha \in \Delta \right\} $ is a family of $%
i $-closed subsets of $X$ satisfies $j$-$\limfunc{int}\left(
\dbigcap\limits_{n\in
\mathbb{N}
}\left( X\setminus U_{\alpha _{n}}\right) \right) \neq \emptyset $ for a
countable subfamily $\left\{ X\setminus U_{\alpha _{n}}:n\in
\mathbb{N}
\right\} $. By hypothesis, the intersection $\dbigcap\limits_{\alpha \in
\Delta }\left( X\setminus U_{\alpha }\right) \neq \emptyset $, and so $%
X\setminus \dbigcup\limits_{\alpha \in \Delta }U_{\alpha }\neq \emptyset $,
i.e., $X\neq \dbigcup\limits_{\alpha \in \Delta }U_{\alpha }$. It is a
contradiction with the fact that $\left\{ U_{\alpha }:\alpha \in \Delta
\right\} $ is an $i$-open cover of $X$. Therefore $X$ is $ij$-weakly Lindel%
\"{o}f.
\end{proof}

\begin{corollary}
A bitoplogical space $X$ is pairwise weakly Lindel\"{o}f if and only if for
every family $\left\{ C_{\alpha }:\alpha \in \Delta \right\} $ by closed
subsets of $X$, there exists a countable subfamily $\left\{ C_{\alpha
_{n}}:n\in
\mathbb{N}
\right\} $ such that $\limfunc{int}\left( \dbigcap\limits_{n\in
\mathbb{N}
}C_{\alpha _{n}}\right) \neq \emptyset $, the intersection $%
\dbigcap\limits_{\alpha \in \Delta }C_{\alpha }\neq \emptyset $.
\end{corollary}

\begin{definition}
A subset $E$ of a bitopological space $X$ is said to be $i$-dense in $%
F\subseteq X$ if $F\subseteq i$-$\limfunc{cl}\left( E\right) $. In
particular, $E$ is $i$-dense in $X$ or is an $i$-dense subset of $X$ if $i$-$%
\limfunc{cl}\left( E\right) =X$. $E$ is said to be dense in $F$ if it is $i$%
-dense in $F$ for each $i=1,2$. In particular, $E$ is dense in $X$ or is a
dense subset of $X$ if it is $i$-dense in $X$ or is an $i$-dense subset of $%
X $ for each $i=1,2$.
\end{definition}

\begin{proposition}
Let $\left( X,\tau _{1},\tau _{2}\right) $ be a bitopological space. For the
following conditions

$\left( i\right) $ $X$ is $ij$-weakly Lindel\"{o}f$;$

$\left( ii\right) $ every $ij$-regular open cover $\left\{ U_{\alpha
}:\alpha \in \Delta \right\} $ of $X$ admits a countable subfamily $\left\{
U_{\alpha _{n}}:n\in
\mathbb{N}
\right\} $ with $j$-dense union in $X;$

$\left( iii\right) $ every family $\left\{ C_{\alpha }:\alpha \in \Delta
\right\} $ of $ij$-regular closed subsets of $X$ such that $%
\dbigcap\limits_{\alpha \in \Delta }C_{\alpha }=\emptyset $ admits a
countable subfamily $\left\{ C_{\alpha _{n}}:n\in
\mathbb{N}
\right\} $ such that $j$-$\limfunc{int}\left( \dbigcap\limits_{n\in
\mathbb{N}
}C_{\alpha _{n}}\right) =\emptyset ;$

we have that $\left( i\right) \Rightarrow \left( ii\right) \Leftrightarrow
\left( iii\right) $ and if $X$ is $ij$-semiregular, then $\left( ii\right)
\Rightarrow \left( i\right) $.
\end{proposition}

\begin{proof}
$\left( i\right) \Rightarrow \left( ii\right) $: It is obvious by the
definition since an $ij$-regular open set is also $i$-open set.

$\left( ii\right) \Leftrightarrow \left( iii\right) $: If $\left\{ C_{\alpha
}:\alpha \in \Delta \right\} $ is a family of $ij$-regular closed subsets of
$X$ such that $\dbigcap\limits_{\alpha \in \Delta }C_{\alpha }=\emptyset $,
then $X=X\setminus \dbigcap\limits_{\alpha \in \Delta }C_{\alpha
}=\dbigcup\limits_{\alpha \in \Delta }\left( X\setminus C_{\alpha }\right) $%
, i.e., the family $\left\{ X\setminus C_{\alpha }:\alpha \in \Delta
\right\} $ is an $ij$-regular open cover of $X$. By $\left( ii\right) $,
there exists a countable subfamily $\left\{ X\setminus C_{\alpha _{n}}:n\in
\mathbb{N}
\right\} $ with $j$-dense union in $X$, i.e., $X=j$-$\limfunc{cl}\left(
\dbigcup\limits_{n\in
\mathbb{N}
}\left( X\setminus C_{\alpha _{n}}\right) \right) $. So $X\setminus j$-$%
\limfunc{cl}\left( \dbigcup\limits_{n\in
\mathbb{N}
}\left( X\setminus C_{\alpha _{n}}\right) \right) =\emptyset $, i.e.,
\begin{equation*}
j\text{-}\limfunc{int}\left( X\setminus \dbigcup\limits_{n\in
\mathbb{N}
}\left( X\setminus C_{\alpha _{n}}\right) \right) =\emptyset .
\end{equation*}
Thus $j$-$\limfunc{int}\left( \dbigcap\limits_{n\in
\mathbb{N}
}C_{\alpha _{n}}\right) =\emptyset $. Conversely, let $\left\{ U_{\alpha
}:\alpha \in \Delta \right\} $ be an $ij$-regular open cover of $X$. Then $%
X=\dbigcup\limits_{\alpha \in \Delta }U_{\alpha }$ and $\left\{ X\setminus
U_{\alpha }:\alpha \in \Delta \right\} $ is a family of $ij$-regular closed
subsets of $X$. Hence $X\setminus \dbigcup\limits_{\alpha \in \Delta
}U_{\alpha }=\emptyset $, i.e., $\dbigcap\limits_{\alpha \in \Delta }\left(
X\setminus U_{\alpha }\right) =\emptyset $. By $\left( iii\right) $, there
exists a countable subfamily $\left\{ X\setminus U_{\alpha _{n}}:n\in
\mathbb{N}
\right\} $ such that $j$-$\limfunc{int}\left( \dbigcap\limits_{n\in
\mathbb{N}
}\left( X\setminus U_{\alpha _{n}}\right) \right) =\emptyset $. So $%
X=X\setminus j$-$\limfunc{int}\left( \dbigcap\limits_{n\in
\mathbb{N}
}\left( X\setminus U_{\alpha _{n}}\right) \right) =j$-$\limfunc{cl}\left(
X\setminus \dbigcap\limits_{n\in
\mathbb{N}
}\left( X\setminus U_{\alpha _{n}}\right) \right) =j$-$\limfunc{cl}\left(
\dbigcup\limits_{n\in
\mathbb{N}
}U_{\alpha _{n}}\right) $ and $\left( ii\right) $ proved.

$\left( ii\right) \Rightarrow \left( i\right) $: Let $\left\{ U_{\alpha
}:\alpha \in \Delta \right\} $ be an $i$-open cover of $X$. Since $X$ is $ij$%
-semiregular, we can assume that $U_{\alpha }$ is $ij$-regular open set for
each $\alpha $. By $\left( ii\right) $, there exists a countable subfamily $%
\left\{ U_{\alpha _{n}}:n\in
\mathbb{N}
\right\} $ with $j$-dense union in $X$, i.e., $X=j$-$\limfunc{cl}\left(
\dbigcup\limits_{n\in
\mathbb{N}
}U_{\alpha _{n}}\right) $. Thus $X$ is $ij$-weakly Lindel\"{o}f and this
completes the proof.
\end{proof}

\begin{corollary}
Let $\left( X,\tau _{1},\tau _{2}\right) $ be a bitopological space. For the
following conditions

$\left( i\right) $ $X$ is pairwise almost Lindel\"{o}f$;$

$\left( ii\right) $ every pairwise regular open cover $\left\{ U_{\alpha
}:\alpha \in \Delta \right\} $ of $X$ admits a countable subfamily $\left\{
U_{\alpha _{n}}:n\in
\mathbb{N}
\right\} $ with dense union in $X;$

$\left( iii\right) $ every family $\left\{ C_{\alpha }:\alpha \in \Delta
\right\} $ of pairwise regular closed subsets of $X$ such that $%
\dbigcap\limits_{\alpha \in \Delta }C_{\alpha }=\emptyset $ admits a
countable subfamily $\left\{ C_{\alpha _{n}}:n\in
\mathbb{N}
\right\} $ such that $\limfunc{int}\left( \dbigcap\limits_{n\in
\mathbb{N}
}C_{\alpha _{n}}\right) =\emptyset ;$ we have that $\left(
i\right) \Rightarrow \left( ii\right) \Leftrightarrow \left(
iii\right) $ and if $X$ is pairwise semiregular, then $\left(
ii\right) \Rightarrow \left( i\right) $.\medskip
\end{corollary}

\noindent It is very obvious that every $ij$-almost Lindel\"{o}f space is $ij$%
-weakly Lindel\"{o}f since%
\begin{equation*}
\dbigcup\limits_{n\in
\mathbb{N}
}j\text{-}\limfunc{cl}\left( U_{\alpha _{n}}\right) \subseteq j\text{-}%
\limfunc{cl}\left( \dbigcup\limits_{n\in
\mathbb{N}
}U_{\alpha _{n}}\right) .
\end{equation*}

\noindent \textbf{Question 1. \ }Is $ij$-weakly Lindel\"{o}f property imply $ij$%
-almost Lindel\"{o}f property?

\noindent The authors expected that the answer of this question is
no.\medskip

\begin{definition}
A bitopological space $X$ is said to be $i$-separable if there exists a
countable $i$-dense subset of $X$. $X$ is said separable if it is $i$%
-separable for each $i=1,2$.
\end{definition}

\begin{proposition}
If the bitopological space $X$ is $j$-separable, then it is $ij$-weakly
Lindel\"{o}f.
\end{proposition}

\begin{proof}
Let $\mathcal{U}=\left\{ U_{\alpha }:\alpha \in \Delta \right\} $ be an $i$%
-open cover of the $j$-separable space $X$. Then $X$ has a countable $j$%
-dense subset $D=\left\{ x_{1},x_{2},\ldots ,x_{n},\ldots \right\} $. Now,
for every $x_{k}\in D$, there exists $\alpha _{k}\in \Delta $ with $x_{k}\in
U_{\alpha _{k}}$. So $X=j$-$\limfunc{cl}\left( D\right) =j$-$\limfunc{cl}%
\left( \dbigcup\limits_{k\in
\mathbb{N}
}\left\{ x_{k}\right\} \right) =j$-$\limfunc{cl}\left( \dbigcup\limits_{k\in
\mathbb{N}
}U_{\alpha _{k}}\right) $. This shows that $X$ is $ij$-weakly Lindel\"{o}f.
\end{proof}

\begin{corollary}
If the bitopological space $X$ is separable, then it is pairwise weakly
Lindel\"{o}f.
\end{corollary}

\begin{definition}
A bitopological space $X$ is called $ij$-weak $P$-space if for each
countable family $\left\{ U_{n}:n\in
\mathbb{N}
\right\} $ of $i$-open sets in $X$, we have $j$-$\limfunc{cl}\left(
\dbigcup\limits_{n\in
\mathbb{N}
}U_{\alpha _{n}}\right) =\dbigcup\limits_{n\in
\mathbb{N}
}j$-$\limfunc{cl}\left( U_{\alpha _{n}}\right) $. $X$ is called pairwise
weak $P$-space if it is both $ij$-weak $P$-space and $ji$-weak $P$-space.
\end{definition}

\begin{proposition}
In $ij$-weak $P$-spaces, $ij$-almost Lindel\"{o}f property is equivalent to $%
ij$-weakly Lindel\"{o}f property.
\end{proposition}

\begin{proof}
The proof follows immediately from the fact that in $ij$-weak $P$-spaces, $%
\dbigcup\limits_{n\in
\mathbb{N}
}j$-$\limfunc{cl}\left( U_{\alpha _{n}}\right) =j$-$\limfunc{cl}\left(
\dbigcup\limits_{n\in
\mathbb{N}
}U_{\alpha _{n}}\right) $ for any countable family $\left\{ U_{n}:n\in
\mathbb{N}
\right\} $ of $i$-open sets in $X$.
\end{proof}

\begin{corollary}
In pairwise weak $P$-spaces, pairwise almost Lindel\"{o}f property is
equivalent to pairwise weakly Lindel\"{o}f property.
\end{corollary}

\begin{lemma}[see \protect\cite{AdemBidi}]
An $ij$-almost regular space is $ij$-almost Lindel\"{o}f if and only if it
is $ij$-nearly Lindel\"{o}f.
\end{lemma}

\begin{proposition}
An $ij$-weakly Lindel\"{o}f, $ij$-almost regular and $ij$-weak $P$-space is $%
ij$-nearly Lindel\"{o}f.
\end{proposition}

\begin{proof}
This is a direct consequence of Proposition 5\ and Lemma 1.
\end{proof}

\begin{corollary}
A pairwise weakly Lindel\"{o}f, pairwise almost regular and pairwise weak $P$%
-space is pairwise nearly Lindel\"{o}f.
\end{corollary}

\begin{lemma}[see \protect\cite{AdemBidi}]
An $ij$-regular space is $ij$-almost Lindel\"{o}f if and only if it is $i$%
-Lindel\"{o}f.
\end{lemma}

\begin{proposition}
An $ij$-weakly Lindel\"{o}f, $ij$-regular and $ij$-weak $P$-space is $i$%
-Lindel\"{o}f.
\end{proposition}

\begin{proof}
This is a direct consequence of Proposition 5 and Lemma 2.
\end{proof}

\begin{corollary}
A pairwise weakly Lindel\"{o}f, pairwise regular and pairwise weak $P$-space
is Lindel\"{o}f.
\end{corollary}

\begin{definition}[see \protect\cite{CammaSan96}, \protect\cite{Eng77}]
Let $X$ be a space. A cover $\mathcal{V}=\left\{ V_{j}:j\in J\right\} $ of $%
X $ is a refinement of another cover $\mathcal{U}=\left\{ U_{\alpha }:\alpha
\in \Delta \right\} $ if for each $j\in J$, there exists an $\alpha \left(
j\right) \in \Delta $ such that $V_{j}\subseteq U_{\alpha \left( j\right) }$%
, i.e., each $V\in \mathcal{V}$ is contained in some $U\in \mathcal{U}$.
\end{definition}

\begin{definition}[see \protect\cite{CammaSan96}, \protect\cite{Eng77}]
A family $\mathcal{U}=\left\{ U_{\alpha }:\alpha \in \Delta \right\} $ of
subsets of a space $X$ is locally finite if for every point $x\in X$, there
exists a neighbourhood $U_{x}$ of $x$ such that the set $\left\{ \alpha \in
\Delta :U_{x}\cap U_{\alpha }\neq \emptyset \right\} $ is finite, i.e., each
$x\in X$ has a neighbourhood $U_{x}$ meeting only finitely many $U\in
\mathcal{U}$.
\end{definition}

\begin{definition}
A bitopological space $X$ is said to be $ij$-nearly paracompact if every
cover of $X$ by $ij$-regular open sets admits a locally finite refinement.
\end{definition}

\begin{proposition}
An $ij$-weakly Lindel\"{o}f, $ij$-semiregular and $ij$-nearly paracompact
bitopological space $X$ is $ij$-almost Lindel\"{o}f.
\end{proposition}

\begin{proof}
Let $\left\{ U_{\alpha }:\alpha \in \Delta \right\} $ be an $ij$-regular
open cover of $X$. Since $X$ is $ij$-nearly paracompact, this cover admits a
locally finite refinement $\left\{ V_{\lambda }:\lambda \in \Lambda \right\}
$. Since $X$ is $ij$-weakly Lindel\"{o}f, there exists a countable subfamily
$\left\{ V_{\lambda _{n}}:n\in
\mathbb{N}
\right\} $ such that $X=j$-$\limfunc{cl}\left( \dbigcup\limits_{n\in
\mathbb{N}
}V_{\lambda _{n}}\right) $. Since the family $\left\{ V_{\lambda _{n}}:n\in
\mathbb{N}
\right\} $ is locally finite, then $j$-$\limfunc{cl}\left(
\dbigcup\limits_{n\in
\mathbb{N}
}V_{\lambda _{n}}\right) =\dbigcup\limits_{n\in
\mathbb{N}
}j$-$\limfunc{cl}\left( V_{\lambda _{n}}\right) $ (see \cite{Eng77}).
Choosing, for $n\in
\mathbb{N}
$, $\alpha _{n}\in \Delta $ such that $V_{\lambda _{n}}\subseteq U_{\alpha
_{n}}$, we obtain $X=\dbigcup\limits_{n\in
\mathbb{N}
}j$-$\limfunc{cl}\left( V_{\lambda _{n}}\right) =\dbigcup\limits_{n\in
\mathbb{N}
}j$-$\limfunc{cl}\left( U_{\alpha _{n}}\right) $. By Proposition 3, $X$ is $%
ij$-almost Lindel\"{o}f.
\end{proof}

\begin{corollary}
A pairwise weakly Lindel\"{o}f, pairwise semiregular and pairwise nearly
paracompact bitopological space $X$ is pairwise almost Lindel\"{o}f.
\end{corollary}

\begin{proposition}
An $ij$-weakly Lindel\"{o}f, $ij$-regular and $ij$-nearly paracompact
bitopological space $X$ is $i$-Lindel\"{o}f.
\end{proposition}

\begin{proof}
Let $\left\{ U_{\alpha }:\alpha \in \Delta \right\} $ be an $i$-open cover
of $X$. Since $X$ is $ij$-regular, then it is $ij$-semiregular. So $X$ is $%
ij $-almost Lindel\"{o}f by Proposition 8. Since $X$ is $ij$-regular, then $X
$ is $i$-Lindel\"{o}f by Lemma 2. \ \ \ \ \ \ \ \ \ \ \ \
\end{proof}

\begin{corollary}
A pairwise weakly Lindel\"{o}f, pairwise regular and pairwise nearly
paracompact bitopological space $X$ is Lindel\"{o}f.
\end{corollary}

\section{\textbf{Pairwise Weakly Lindel\"{o}f Subspaces and Subsets}}

\noindent A subset $S$ of a bitopological space $X$ is said to be $ij$-weakly Lindel%
\"{o}f (resp. pairwise weakly Lindel\"{o}f) if $S$ is $ij$-weakly Lindel\"{o}%
f (resp. pairwise weakly Lindel\"{o}f) as a subspace of $X$, i.e., $S$ is $%
ij $-weakly Lindel\"{o}f (resp. pairwise weakly Lindel\"{o}f) with respect
to the inducted bitopology from the bitopology of $X$.

\begin{definition}
A subset $S$ of a bitopological space $X$ is said to be $ij$-weakly Lindel%
\"{o}f relative to $X$ if for every cover $\left\{ U_{\alpha }:\alpha \in
\Delta \right\} $ of $S$ by $i$-open subsets of $X$ such that $S\subseteq
\dbigcup\limits_{\alpha \in \Delta }U_{\alpha }$, there exists a countable
subset $\left\{ \alpha _{n}:n\in
\mathbb{N}
\right\} $ of $\Delta $ such that $S\subseteq j$-$\limfunc{cl}\left(
\dbigcup\limits_{n\in
\mathbb{N}
}U_{\alpha _{n}}\right) $. $S$ is said pairwise weakly Lindel\"{o}f relative
to $X$ if it is both $ij$-weakly Lindel\"{o}f relative to $X$ and $ji$%
-weakly Lindel\"{o}f relative to $X$.
\end{definition}

\begin{proposition}
Let $S$ be a subset of a bitopological space $X$. Then $S$ is $ij$-weakly
Lindel\"{o}f relative to $X$ if and only if for every family $\left\{
C_{\alpha }:\alpha \in \Delta \right\} $ of $i$-closed subsets of $X$ such
that $\left( \dbigcap\limits_{\alpha \in \Delta }C_{\alpha }\right) \cap
S=\emptyset $, there exists a countable subfamily $\left\{ C_{\alpha
_{n}}:n\in
\mathbb{N}
\right\} $ such that $j$-$\limfunc{int}\left( \dbigcap\limits_{n\in
\mathbb{N}
}C_{\alpha _{n}}\right) \cap S=\emptyset $.
\end{proposition}

\begin{proof}
Let $\left\{ C_{\alpha }:\alpha \in \Delta \right\} $ is a family of $i$%
-closed subsets of $X$ such that $\left( \dbigcap\limits_{\alpha \in \Delta
}C_{\alpha }\right) \cap S=\emptyset $. Then $S\subseteq X\setminus \left(
\dbigcap\limits_{\alpha \in \Delta }C_{\alpha }\right)
=\dbigcup\limits_{\alpha \in \Delta }\left( X\setminus C_{\alpha }\right) $,
so $\left\{ X\setminus C_{\alpha }:\alpha \in \Delta \right\} $ forms a
family of $i$-open subsets of $X$ covering $S$. By hypothesis, there exists
a countable subfamily $\left\{ X\setminus C_{\alpha _{n}}:n\in
\mathbb{N}
\right\} $ such that $S\subseteq j$-$\limfunc{cl}\left(
\dbigcup\limits_{n\in
\mathbb{N}
}\left( X\setminus C_{\alpha _{n}}\right) \right) $. Hence $\left(
X\setminus j\text{-}\limfunc{cl}\left( \dbigcup\limits_{n\in
\mathbb{N}
}\left( X\setminus C_{\alpha _{n}}\right) \right) \right) \cap S=\emptyset $%
, i.e., $j$-$\limfunc{int}\left( X\setminus \dbigcup\limits_{n\in
\mathbb{N}
}\left( X\setminus C_{\alpha _{n}}\right) \right) \cap S=\emptyset $. Thus $%
j $-$\limfunc{int}\left( \dbigcap\limits_{n\in
\mathbb{N}
}C_{\alpha _{n}}\right) \cap S=\emptyset $. Conversely, let $\left\{
U_{\alpha }:\alpha \in \Delta \right\} $ be a family of $i$-open subsets in $%
X$ such that $S\subseteq \dbigcup\limits_{\alpha \in \Delta }U_{\alpha }$.
Then $\left( X\setminus \dbigcup\limits_{\alpha \in \Delta }U_{\alpha
}\right) \cap S=\emptyset $, i.e., $\left( \dbigcap\limits_{\alpha \in
\Delta }X\setminus U_{\alpha }\right) \cap S=\emptyset $. Since $\left\{
X\setminus U_{\alpha }:\alpha \in \Delta \right\} $ be a family of $i$%
-closed subsets of $X$, by hypothesis there exists a countable subfamily $%
\left\{ X\setminus U_{\alpha _{n}}:n\in
\mathbb{N}
\right\} $ such that $j$-$\limfunc{int}\left( \dbigcap\limits_{n\in
\mathbb{N}
}\left( X\setminus U_{\alpha _{n}}\right) \right) \cap S=\emptyset $.
Therefore $S\subseteq X\setminus j$-$\limfunc{int}\left(
\dbigcap\limits_{n\in
\mathbb{N}
}\left( X\setminus U_{\alpha _{n}}\right) \right) =j$-$\limfunc{cl}\left(
X\setminus \left( \dbigcap\limits_{n\in
\mathbb{N}
}\left( X\setminus U_{\alpha _{n}}\right) \right) \right) =j$-$\limfunc{cl}%
\left( \dbigcup\limits_{n\in
\mathbb{N}
}U_{\alpha _{n}}\right) $. This completes the proof.
\end{proof}

\begin{corollary}
Let $S$ be a subset of a bitopological space $X$. Then $S$ is pairwise
weakly Lindel\"{o}f relative to $X$ if and only if for every family $\left\{
C_{\alpha }:\alpha \in \Delta \right\} $ of closed subsets of $X$ such that $%
\left( \dbigcap\limits_{\alpha \in \Delta }C_{\alpha }\right) \cap
S=\emptyset $, there exists a countable subfamily $\left\{ C_{\alpha
_{n}}:n\in
\mathbb{N}
\right\} $ such that $\limfunc{int}\left( \dbigcap\limits_{n\in
\mathbb{N}
}C_{\alpha _{n}}\right) \cap S=\emptyset $.
\end{corollary}

\begin{proposition}
Let $X$ be a bitopological space and $S\subseteq X$. For the following
conditions

$\left( i\right) $ $S$ is $ij$-weaklyLindel\"{o}f relative to $X;$

$\left( ii\right) $ every family by $ij$-regular open subsets $\left\{
U_{\alpha }:\alpha \in \Delta \right\} $ of $X$ that cover $S$ admits a
countable subfamily $\left\{ U_{\alpha _{n}}:n\in
\mathbb{N}
\right\} $ with $j$-dense union in $S;$

$\left( iii\right) $ every family $\left\{ C_{\alpha }:\alpha \in \Delta
\right\} $ of $ij$-regular closed subsets of $X$ such that%
\begin{equation*}
\left( \dbigcap\limits_{\alpha \in \Delta }C_{\alpha }\right) \cap
S=\emptyset
\end{equation*}
admits a countable subfamily $\left\{ C_{\alpha _{n}}:n\in
\mathbb{N}
\right\} $ such that $j$-$\limfunc{int}\left( \dbigcap\limits_{n\in
\mathbb{N}
}C_{\alpha _{n}}\right) \cap S=\emptyset ;$ we have that $\left(
i\right) \Rightarrow \left( ii\right) \Leftrightarrow \left(
iii\right) $ and if $X$ is $ij$-semiregular, then $\left(
ii\right) \Rightarrow \left( i\right) $.
\end{proposition}

\begin{proof}
$\left( i\right) \Rightarrow \left( ii\right) $: It is obvious by the
definition since an $ij$-regular open set is also $i$-open set.

$\left( ii\right) \Leftrightarrow \left( iii\right) $: If $\left\{ C_{\alpha
}:\alpha \in \Delta \right\} $ is a family of $ij$-regular closed subsets of
$X$ such that $\left( \dbigcap\limits_{\alpha \in \Delta }C_{\alpha }\right)
\cap S=\emptyset $ , then $S\subseteq X\setminus $ $\dbigcap\limits_{\alpha
\in \Delta }C_{\alpha }=\dbigcup\limits_{\alpha \in \Delta }\left(
X\setminus C_{\alpha }\right) $, i.e., the family $\left\{ X\setminus
C_{\alpha }:\alpha \in \Delta \right\} $ is an $ij$-regular open subsets of $%
X$ that cover $S$. By $\left( ii\right) $, there exists a countable
subfamily $\left\{ X\setminus C_{\alpha _{n}}:n\in
\mathbb{N}
\right\} $ such that $S\subseteq j$-$\limfunc{cl}\left(
\dbigcup\limits_{n\in
\mathbb{N}
}\left( X\setminus C_{\alpha _{n}}\right) \right) =X\setminus j$-$\limfunc{%
int}\left( X\setminus \dbigcup\limits_{n\in
\mathbb{N}
}\left( X\setminus C_{\alpha _{n}}\right) \right) =X\setminus j$-$\limfunc{%
int}\left( \dbigcap\limits_{n\in
\mathbb{N}
}C_{\alpha _{n}}\right) $. So, $j$-$\limfunc{int}\left(
\dbigcap\limits_{n\in
\mathbb{N}
}C_{\alpha _{n}}\right) \cap S=\emptyset $. Conversely, let $\left\{
U_{\alpha }:\alpha \in \Delta \right\} $ be a family of $ij$-regular open
subsets of $X$ that cover $S$. Then $S\subseteq \dbigcup\limits_{\alpha \in
\Delta }U_{\alpha }$ and $\left\{ X\setminus U_{\alpha }:\alpha \in \Delta
\right\} $ is a family of $ij$-regular closed subsets of $X$. Hence $\left(
X\setminus \dbigcup\limits_{\alpha \in \Delta }U_{\alpha }\right) \cap
S=\emptyset $, i.e., $\left( \dbigcap\limits_{\alpha \in \Delta }\left(
X\setminus U_{\alpha }\right) \right) \cap S=\emptyset $. By $\left(
iii\right) $, there exists a countable subfamily $\left\{ X\setminus
U_{\alpha _{n}}:n\in
\mathbb{N}
\right\} $ such that $j$-$\limfunc{int}\left( \dbigcap\limits_{n\in
\mathbb{N}
}\left( X\setminus U_{\alpha _{n}}\right) \right) \cap S=\emptyset $. So $%
S\subseteq X\setminus j$-$\limfunc{int}\left( \dbigcap\limits_{n\in
\mathbb{N}
}\left( X\setminus U_{\alpha _{n}}\right) \right) =j$-$\limfunc{cl}\left(
X\setminus \dbigcap\limits_{n\in
\mathbb{N}
}\left( X\setminus U_{\alpha _{n}}\right) \right) =j$-$\limfunc{cl}\left(
\dbigcup\limits_{n\in
\mathbb{N}
}U_{\alpha _{n}}\right) $.

$\left( ii\right) \Rightarrow \left( i\right) $: Let $\left\{ U_{\alpha
}:\alpha \in \Delta \right\} $ be a family of $i$-open subsets of $X$ that
cover $S$. Since $X$ is $ij$-semiregular, we can assume that $U_{\alpha }$
is $ij$-regular open set for each $\alpha $. By $\left( ii\right) $, there
exists a countable subfamily $\left\{ U_{\alpha _{n}}:n\in
\mathbb{N}
\right\} $ such that $S\subseteq j$-$\limfunc{cl}\left(
\dbigcup\limits_{n\in
\mathbb{N}
}U_{\alpha _{n}}\right) $. This completes the proof.
\end{proof}

\begin{corollary}
Let $X$ be a bitopological space and $S\subseteq X$. For the following
conditions

$\left( i\right) $ $S$ is pairwise weaklyLindel\"{o}f relative to $X;$

$\left( ii\right) $ every family by pairwise regular open subsets $\left\{
U_{\alpha }:\alpha \in \Delta \right\} $ of $X$ that cover $S$ admits a
countable subfamily $\left\{ U_{\alpha _{n}}:n\in
\mathbb{N}
\right\} $ with dense union in $S;$

$\left( iii\right) $ every family $\left\{ C_{\alpha }:\alpha \in \Delta
\right\} $ of pairwise regular closed subsets of $X$ such that $\left(
\dbigcap\limits_{\alpha \in \Delta }C_{\alpha }\right) \cap S=\emptyset $
admits a countable subfamily $\left\{ C_{\alpha _{n}}:n\in
\mathbb{N}
\right\} $ such that $\limfunc{int}\left( \dbigcap\limits_{n\in
\mathbb{N}
}C_{\alpha _{n}}\right) \cap S=\emptyset ;$ we have that $\left(
i\right) \Rightarrow \left( ii\right) \Leftrightarrow \left(
iii\right) $ and if $X$ is pairwise semiregular, then $\left(
ii\right) \Rightarrow \left( i\right) $.
\end{corollary}

\begin{proposition}
Let $X$ be a bitopological space and $A$ be any subset of $X$. If $A$ is $ij$%
-weakly Lindel\"{o}f, then it is $ij$-weakly Lindel\"{o}f relative to $X$.
\end{proposition}

\begin{proof}
Let $\left\{ U_{\alpha }:\alpha \in \Delta \right\} $ be a family of $i$%
-open subsets of $X$ that cover $A$. Then for each $\alpha $, we can find an
$i$-open set $V_{\alpha }$ of $A$ with $U_{\alpha }\cap A=V_{\alpha }$. Thus
$\left\{ V_{\alpha }:\alpha \in \Delta \right\} $ be a cover of $A$ by $i$%
-open subsets of $A$. Since $A$ is $ij$-weakly Lindel\"{o}f, then there
exists a countable subset $\left\{ \alpha _{n}:n\in
\mathbb{N}
\right\} $ of $\Delta $ such that $A=j$-$\limfunc{cl}_{A}\left(
\dbigcup\limits_{n\in
\mathbb{N}
}V_{\alpha _{n}}\right) \subseteq j$-$\limfunc{cl}_{X}\left(
\dbigcup\limits_{n\in
\mathbb{N}
}U_{\alpha _{n}}\right) $. Therefore $A$ is $ij$-weakly Lindel\"{o}f
relative to $X$.
\end{proof}

\begin{corollary}
Let $X$ be a bitopological space and $A$ be any subset of $X$. If $A$ is
pairwise weakly Lindel\"{o}f, then it is pairwise weakly Lindel\"{o}f
relative to $X$.
\end{corollary}

\noindent \textbf{Question 2. }Is the converse of Proposition 12
above true?

\noindent The authors expected that the answer is false.\medskip

\noindent The converse of Proposition 12 hold if $A\subseteq X$ is
$i$-open. We prove that in the following proposition.

\begin{proposition}
Let $X$ be a bitopological space and $A$ an $i$-open subset of $X$. Then $A$
is $ij$-weakly Lindel\"{o}f if and only if it is $ij$-weakly Lindel\"{o}f
relative to $X$.
\end{proposition}

\begin{proof}
The proof of necessity can be obtained from Proposition 12. For the
sufficiency, let $\left\{ U_{\alpha }:\alpha \in \Delta \right\} $ be an $i$%
-open cover of $A$. Since $i$-open subsets of an $i$-open subspace of $X$ is
$i$-open in $X$, then $\left\{ U_{\alpha }:\alpha \in \Delta \right\} $ is a
cover of $A$ by $i$-open subsets of $X$. Since $A$ is $ij$-weakly Lindel\"{o}%
f relative to $X$, there exists a countable subset $\left\{ \alpha _{n}:n\in
\mathbb{N}
\right\} $ of $\Delta $ such that $A\subseteq j$-$\limfunc{cl}_{X}\left(
\dbigcup\limits_{n\in
\mathbb{N}
}U_{\alpha _{n}}\right) $. Then $A=j$-$\limfunc{cl}_{X}\left(
\dbigcup\limits_{n\in
\mathbb{N}
}U_{\alpha _{n}}\right) \cap A=j$-$\limfunc{cl}_{X}\left( \left(
\dbigcup\limits_{n\in
\mathbb{N}
}U_{\alpha _{n}}\right) \cap A\right) =j$-$\limfunc{cl}_{X}\left(
\dbigcup\limits_{n\in
\mathbb{N}
}\left( U_{\alpha _{n}}\cap A\right) \right) =j$-$\limfunc{cl}_{A}\left(
\dbigcup\limits_{n\in
\mathbb{N}
}U_{\alpha _{n}}\right) $. Therefore $A$ is $ij$-weakly Lindel\"{o}f.
\end{proof}

\begin{corollary}
Let $X$ be a bitopological space and $A$ an open subset of $X$. Then $A$ is
pairwise weakly Lindel\"{o}f if and only if it is pairwise weakly Lindel\"{o}%
f relative to $X$.
\end{corollary}

\begin{remark}
The above result shows that in an $i$-open subset of a bitopological space $%
X $, $ij$-weakly Lindel\"{o}f property and $ij$-weakly Lindel\"{o}f relative
to $X$ property are equivalent.
\end{remark}

\begin{remark}
The space $X$ in Proposition $12$, Proposition $13$, Corollary $12$ and
Corollary $13$ is any bitopological space.
\end{remark}

\noindent If we consider $X$ itself is an $ij$-weakly Lindel\"{o}f
space, we have the following results.

\begin{proposition}
Every $ij$-regular closed and $j$-open subset of an $ij$-weakly Lindel\"{o}f
and $ij$-semiregular space $X$ is $ij$-weakly Lindel\"{o}f relative to $X$.
\end{proposition}

\begin{proof}
Let $A$ be an $ij$-regular closed and $j$-open subset of $X$. If $\left\{
U_{\alpha }:\alpha \in \Delta \right\} $ is a cover of $A$ by $ij$-regular
open subsets of $X$, then $X=\left( \underset{\alpha \in \Delta }{\bigcup }%
U_{\alpha }\right) \cup \left( X\setminus A\right) $. Hence the family $%
\left\{ U_{\alpha }:\alpha \in \Delta \right\} \cup \left\{ X\setminus
A\right\} $ forms an $ij$-regular open cover of $X$. Since $X$ is $ij$%
-weakly Lindel\"{o}f, there will be a countable subfamily, $\left\{
X\setminus A,U_{\alpha _{1}},U_{\alpha _{2}},\ldots \right\} $ such that%
\begin{eqnarray*}
X &=&j\text{-}\limfunc{cl}\left( \dbigcup\limits_{n\in
\mathbb{N}
}\left( U_{\alpha _{n}}\cup X\setminus A\right) \right) \\
&=&j\text{-}\limfunc{cl}\left( \left( \dbigcup\limits_{n\in
\mathbb{N}
}U_{\alpha _{n}}\right) \cup X\setminus A\right) \\
&=&j\text{-}\limfunc{cl}\left( \dbigcup\limits_{n\in
\mathbb{N}
}U_{\alpha _{n}}\right) \cup j\text{-}\limfunc{cl}\left( X\setminus A\right)
\\
&=&j\text{-}\limfunc{cl}\left( \dbigcup\limits_{n\in
\mathbb{N}
}U_{\alpha _{n}}\right) \cup \left( X\setminus A\right)
\end{eqnarray*}%
by Proposition 3. But $A$ and $X\setminus A$ are disjoint; hence $A\subseteq
j$-$\limfunc{cl}\left( \dbigcup\limits_{n\in
\mathbb{N}
}U_{\alpha _{n}}\right) $. This shows that $A$ is $ij$-weakly Lindel\"{o}f
relative to $X$ and completes the proof.
\end{proof}

\begin{corollary}
Every pairwise regular closed and open subset of a pairwise weakly Lindel%
\"{o}f and pairwise semiregular space $X$ is pairwise weakly Lindel\"{o}f
relative to $X$.
\end{corollary}

\begin{proposition}
An $i$-clopen and $j$-open subset of an $ij$-weakly Lindel\"{o}f space $X$
is $ij$-weakly Lindel\"{o}f.
\end{proposition}

\begin{proof}
Let $F$ be an $i$-clopen and $j$-open subset of $X$. By Proposition 13
above, it is sufficient to prove that $F$ is $ij$-almost Lindel\"{o}f
relative to $X$. Let $\left\{ U_{\alpha }:\alpha \in \Delta \right\} $ be a
family of $i$-open subsets of $X$ that cover $F$. Then $\left\{ U_{\alpha
}:\alpha \in \Delta \right\} \cup \left\{ X\setminus F\right\} $ forms an $i$%
-open cover of $X$. Since $X$ is $ij$-weakly Lindel\"{o}f, there exists a
countable subset $\left\{ \alpha _{n}:n\in
\mathbb{N}
\right\} $ of $\Delta $ such that%
\begin{eqnarray*}
X &=&j\text{-}\limfunc{cl}\left( \dbigcup\limits_{n\in
\mathbb{N}
}\left( U_{\alpha _{n}}\cup X\setminus F\right) \right) \\
&=&j\text{-}\limfunc{cl}\left( \left( \dbigcup\limits_{n\in
\mathbb{N}
}U_{\alpha _{n}}\right) \cup \left( X\setminus F\right) \right) \\
&=&j\text{-}\limfunc{cl}\left( \dbigcup\limits_{n\in
\mathbb{N}
}U_{\alpha _{n}}\right) \cup j\text{-}\limfunc{cl}\left( X\setminus F\right)
\\
&=&j\text{-}\limfunc{cl}\left( \dbigcup\limits_{n\in
\mathbb{N}
}U_{\alpha _{n}}\right) \cup \left( X\setminus F\right) .
\end{eqnarray*}%
But $F$ and $X\setminus F$ are disjoint; hence $F\subseteq j$-$\limfunc{cl}%
\left( \dbigcup\limits_{n\in
\mathbb{N}
}U_{\alpha _{n}}\right) $. This shows that $F$ is $ij$-weakly Lindel\"{o}f
relative to $X$ and completes the proof.
\end{proof}

\begin{corollary}
A clopen subset of a pairwise weakly Lindel\"{o}f space $X$ is pairwise
weakly Lindel\"{o}f.
\end{corollary}

\noindent \textbf{Question 3.} \ Is $i$-closed subset of an
$ij$-weakly Lindel\"{o}f space $X$ is $ij$-weakly Lindel\"{o}f?

\noindent \textbf{Question 4.} \ Is $ij$-regular open subset of an $ij$-weakly Lindel%
\"{o}f space $X$ is $ij$-weakly Lindel\"{o}f?

\noindent The authors expected that the answer of both questions
are false.\medskip

\noindent Observe that, the condition in Proposition 14 that a subset should be $ij$%
-regular closed and $j$-open and in Proposition 15 that a subset should be $%
i $-clopen and $j$-open are necessary and it is not sufficient to be only $i$%
-open by example below show. In general, arbitrary subsets of $ij$-wekly
Lindel\"{o}f spaces need not be $ij$-weakly Lindel\"{o}f relative to the
spaces and so not $ij$-weakly Lindel\"{o}f by Proposition 12.

\begin{example}
Let $\Omega $ denotes the set of ordinals which are less than or equal to
the first uncountable ordinal $\omega _{1}$. This $\Omega $ is an
uncountable well-ordered set with a largest element $\omega _{1}$, having
the property that if $\alpha \in \Omega $ with $\alpha <\omega _{1}$, then $%
\left\{ \beta \in \Omega :\beta \leq \alpha \right\} $ is countable. Since $%
\Omega $ is a totally ordered space, it can be provided with its order
topology. Let us denote this order topology by $\tau _{1}$. Let $\mathcal{B}$
be a collection of all sets in $\Omega $ of the form $\left( a,b\right)
=\left\{ \beta \in \Omega :a<\beta <b\right\} ,\left( a_{0},\omega _{1}%
\right] =\left\{ \beta \in \Omega :a_{0}<\beta \leq \omega _{1}\right\} $
and $\left[ 1,b_{0}\right) =\left\{ \beta \in \Omega :1\leq \beta
<b_{0}\right\} $. Then the collection $\mathcal{B}$ is a base for the order
topology $\tau _{1}$ for $\Omega $. Choose discrete topology as another
topology for $\Omega $ denoted by $\tau _{2}$. So $\left( \Omega ,\tau
_{1},\tau _{2}\right) $ form a bitopological space. Now $\Omega $ is a $1$%
-Lindel\"{o}f space $\left( \text{see \cite{Willard70}}\right) $, so it is $%
12$-weakly Lindel\"{o}f. The subset $\Omega _{0}=\Omega \setminus \left\{
\omega _{1}\right\} $, however is not $1$-Lindel\"{o}f $\left( \text{see
\cite{Willard70}}\right) $. We notice that $\Omega _{0}$ is $1$-open subset
of $\Omega $ since $\Omega _{0}=\left[ 1,\omega _{1}\right) =\left\{ \beta
\in \Omega :1\leq \beta <\omega _{1}\right\} $. So $\Omega _{0}$ is not $12$%
-wekly Lindel\"{o}f by Proposition $7$ since it is $12$-regular and $12$%
-weak $P$-space. Moreover $\Omega _{0}$ is not $12$-weakly Lindel\"{o}f
relative to $\Omega $ by Proposition $13$.
\end{example}

\noindent So we can say that in general, an $ij$-weakly
Lindel\"{o}f property is not hereditary property and therefore
pairwise weakly Lindel\"{o}f property is not so.

\begin{definition}
A bitopological space $X$ is said to be hereditary $ij$-weakly Lindel\"{o}f
if every subspace of $X$ is $ij$-weakly Lindel\"{o}f. $X$ is said hereditary
pairwise\ weakly Lindel\"{o}f if it is both hereditary $ij$-weakly Lindel%
\"{o}f and hereditary $ji$-weakly Lindel\"{o}f.
\end{definition}

\begin{proposition}
Let $X$ be an $ij$-semiregular bitopological space. Then $X$ is $i$-open
hereditary $ij$-weakly Lindel\"{o}f if and only if any $A\in \tau _{ij}^{s}$
is $ij$-weakly Lindel\"{o}f.
\end{proposition}

\begin{proof}
Let $X$ be an $ij$-semiregular and $i$-open hereditary $ij$-weakly Lindel\"{o}%
f space. Since $\tau _{ij}^{s}\subseteq \tau _{i}$, it is obvious that any $%
A\in \tau _{ij}^{s}$ implies $A\in \tau _{i}$ and hence $A$ is $ij$-weakly
Lindel\"{o}f. Conversely, let $B\subseteq X$ be an $i$-open subset of $ij$%
-weakly Lindel\"{o}f space $X$. By Proposition 13,\ it is sufficient to
prove that $B$ is $ij$-weakly Lindel\"{o}f relative to $X$. Let $\mathcal{U}%
=\left\{ U_{\alpha }:\alpha \in \Delta \right\} $ be a family by $ij$%
-regular open subsets of $X$ such that $B\subseteq \dbigcup\limits_{\alpha
\in \Delta }U_{\alpha }$. The set $A=\dbigcup\limits_{\alpha \in \Delta
}U_{\alpha }\in \tau _{ij}^{s}$, so by hypothesis $A$ is $ij$-weakly Lindel%
\"{o}f. Hence there exists a countable subfamily $\left\{ U_{\alpha
_{n}}:n\in
\mathbb{N}
\right\} $ of $\mathcal{U}$ such that $A=j$-$\limfunc{cl}\left(
\dbigcup\limits_{n\in
\mathbb{N}
}U_{\alpha _{n}}\right) $ and therefore $B\subseteq j$-$\limfunc{cl}\left(
\dbigcup\limits_{n\in
\mathbb{N}
}U_{\alpha _{n}}\right) $. This completes the proof.
\end{proof}

\begin{corollary}
A bitopological space $X$ is open hereditary pairwise weakly Lindel\"{o}f if
and only if any $A\in \tau ^{s}$ is pairwise weakly Lindel\"{o}f.
\end{corollary}

\end{document}